\documentclass[11pt]{amsart}

\usepackage[a4paper,hmargin=3.5cm,vmargin=3.5cm]{geometry}
\usepackage{amsfonts,amssymb,amscd,amstext}
\usepackage{graphicx}
\usepackage[dvips]{epsfig}

\usepackage{fancyhdr}
\usepackage{times}

\usepackage{enumerate}
\usepackage{titlesec}
\usepackage{mathrsfs}

\titleformat{\section}
{\filcenter\bfseries\large} {\thesection{.}}{0.2cm}{}
\titleformat{\subsection}[runin]
{\bfseries} {\thesubsection{.}}{0.15cm}{}[.]
\titleformat{\subsubsection}[runin]
{\em}{\thesubsubsection{.}}{0.15cm}{}[.]

\usepackage[up,bf]{caption}

\newtheorem{theorem}{Theorem}[section]
\newtheorem{proposition}[theorem]{Proposition}

\newtheorem{lemma}[theorem]{Lemma}
\newtheorem{corollary}[theorem]{Corollary}
\newtheorem{remark}[theorem]{Remark}

\theoremstyle{definition}

\numberwithin{equation}{section}
\numberwithin{figure}{section}

\usepackage{color}

\begin{document}

\fancyhead[CO]{Stable CMC and index one minimal surfaces} 
\fancyhead[CE]{R. Souam} 
\fancyhead[RO,LE]{\thepage} 

\thispagestyle{empty}

\vspace*{1cm}
\begin{center}
{\bf\LARGE Stable  CMC and index one minimal  surfaces }

{\bf \LARGE in conformally flat  manifolds}

\vspace*{0.5cm}

{\large\bf  Rabah Souam}
\end{center}

\footnote[0]{\vspace*{-0.4cm}

\noindent Institut de Math\'{e}matiques de Jussieu-Paris Rive Gauche,   UMR 7586, B\^{a}timent Sophie Germain,  Case 7012, 75205  Paris Cedex 13, France.

\noindent e-mail: {\tt souam@math.jussieu.fr}
}

\vspace*{1cm}

\begin{quote}
{\small
\noindent {\bf Abstract}\hspace*{0.1cm} Let $M$ be a Riemannian 3-manifold  of nonnegative Ricci curvature, Ric $\geq 0.$  Suppose that $M$ is conformally flat and simply connected or more generally that it admits a conformal immersion into the standard 3-sphere. Let $\Sigma$   be a  compact connected and orientable surface immersed  in $M$  which is  a stable constant mean curvature (CMC) surface  or an index one  minimal surface. We prove that $\Sigma$ is homeomorphic either to   a sphere or to a torus. Moreover, in case $\Sigma$ is homeomorphic to a torus, then  it is embedded, minimal, conformal to a flat square torus and Ric$(N)=0$ where $N$ is a unit field normal to $\Sigma.$  The result is sharp, we can perturb the standard metric on the 3-sphere in its conformal class to obtain metrics of nonnegative Ricci curvature  admitting  minimal tori of index one and which are stable as CMC surfaces. 

As a consequence, in any 3-sphere  of positive Ricci curvature which is  conformally flat, the isoperimetric domains  are topologically 3-balls. This proves a special case of a conjecture of Ros \cite{ros0, ros1}. 

\vspace*{0.2cm}

\noindent{\bf Keywords}\hspace*{0.1cm} Constant mean curvature surfaces, minimal surfaces, stability, isoperimetric problem.
\vspace*{0.2cm}

\noindent{\bf Mathematics Subject Classification (2010)}\hspace*{0.1cm} 53C42,  49Q10
}
\end{quote}

\section{Introduction}

Minimal surfaces and constant mean curvature surfaces (CMC)  in Riemannian 3-manifolds are both critical points of  the area functional. The former  are critical for compactly supported variations and  the latter for variations which furthermore keep the 
{\it enclosed volume} constant. They are called stable if they minimize the area up to second order for those variations. In the CMC case, the terminologies volume preserving stable and weakly stable are  sometimes used  to emphasize the difference  with the minimal case. 
Stable minimal surfaces are of fundamental importance in the general theory of minimal surfaces and compact stable CMC surfaces are equally important in studying the isoperimetric problem in Riemannian 3-manifolds since the boundary of an isoperimetric region is a stable CMC.

From the variational viewpoint, the next interesting class of surfaces to consider in the minimal case are those having Morse  index one. Actually, these surfaces  have received much attention, Pitts \cite{pitts} and Pitts and Rubinstein \cite{pitts-rubinstein} have, for instance, constructed 
many examples using minimax methods. 

A positivity assumption on the curvature of the ambient 3-manifold makes the stability condition more tractable and one can indeed control the topology of stable CMC and index one minimal  surfaces. Actually, improving previous results, Ros \cite{ros2} proved the sharp result that  compact orientable stable CMC surfaces and compact orientable  index one  minimal surfaces in  orientable 3-manifolds with nonnegative Ricci curvature have genus $\leq 3.$ Note that, under this assumption, it is an easy fact that a compact orientable minimal surface is stable if and only if it is totally geodesic and the Ricci curvature evaluated on a unit normal to it vanishes. Let us mention also the works by Ritor\'e-Ros \cite{ritore-ros} and Ritor\'e \cite{ritore} about index one minimal surfaces in flat 3-manifolds.

In this paper, we consider these questions in 3-manifolds of nonnegative Ricci curvature which are  conformally flat. More precisely, we consider 3-manifolds with nonnegative Ricci curvature which admit a conformal immersion into the standard 3-sphere $\mathbb S^3.$ Simply connected and conformally flat 3-manifolds are examples of manifolds satisfying the latter condition. Indeed they can be conformally immersed into $\mathbb S^3$ by means of   a developing map. We will prove (Theorem \ref{the result}) that a connected compact  and orientable surface in this kind of manifolds which is a stable CMC  or a minimal surface of index one  is  topologically a sphere or a torus. Moreover, if such a surface is a torus,  then it has to be embedded, minimal, conformal to a flat square torus and Ric$(N)=0$ where $N$ is a unit field normal to the torus
and Ric denotes the Ricci curvature of the ambient manifold. The result is sharp, we are  able to perturb the standard metric on $\mathbb S^3$ in its conformal class to exhibit examples of metrics of nonnegative Ricci curvature admitting minimal  tori of index one which are stable as CMC surfaces (cf. Section \ref{examples}). 

A consequence of our result is a partial  solution to a conjecture of Ros.  In \cite{ros0,ros1}, Ros conjectured that in the 3-sphere endowed with a metric of positive Ricci curvature, the isoperimetric regions are topologically 3-balls.  In this direction, he proved \cite{ros0} that if $M$ is a compact $3-$manifold with  Ricci curvature Ric $ \geq 2$ and volume $V(M)\geq V(\mathbb S^3) /2,$ then any isoperimetric domain of $M$ is bounded by either a sphere or a torus. 
Our main result (Theorem \ref{the result}) has as an immediate corollary that Ros' conjecture is true for conformally flat metrics (Corollary \ref{ros-conjecture}).

\section{Preliminaries}\label{prelim}

 Let $(M,\langle,\rangle)$  be an orientable Riemannian 3-manifold and  $\Sigma$ an  orientable immersed compact CMC 
surface without boundary  in $M.$ We note that when the (constant) mean curvature of $\Sigma$ is not zero then $\Sigma$ is automatically orientable since its mean curvature field is a non vanishing global normal. Call $N$ a unit field normal to $\Sigma.$
We consider on $\Sigma$ the following quadratic form:
$$ Q(u,v)= \int_{\Sigma}\langle \nabla u, \nabla v\rangle - (|\sigma|^2+\text{Ric}(N))uv \,dA=-\int_{\Sigma}
 u Lv\,dA,\quad u,v\in\mathcal{C}^\infty(\Sigma),$$
  where $\nabla u$ stands for the gradient of $u,$   $|\sigma|^2$ is the square of the norm of second fundamental form $\sigma$ of 
the immersion and Ric$(N)$ denotes the Ricci curvature of $M$ evaluated on the  field $N.$  The linear operator 
$L=\Delta + |\sigma|^2+\text{Ric}(N)$ is the Jacobi operator of the surface, $\Delta$ being the Laplacian on $\Sigma.$

 The stability
condition in the minimal case means that the quadratic form $Q$ is nonnegative on $ \mathcal{C}^\infty(\Sigma)$ and in the 
CMC case it means that:
\begin{equation}\label{stability}
Q(u,u)\geq 0,\quad \text{for any}\quad  u\in  \mathcal  {C}^\infty(\Sigma) \quad \text{satisfying} \int_{\Sigma} u\,dA =0.
\end{equation}
 See  \cite{barbosa et al} for the details.

 Let $\Sigma$ be a minimal surface as above. Its index is by definition the number of negative eigenvalues, counted with multiplicities, of its Jacobi operator. So $\Sigma$ is stable if and only if it has index zero. Taking as a test function a constant (non-zero) function, we see immediately that when the Ricci curvature is nonnegative,  a minimal surface $\Sigma$ as above,  is stable as a minimal surface if and only if it is  totally geodesic and  Ric$(N)=0.$  

In the sequel we denote by $\mathbb S^n$ the $n$-dimensional sphere endowed with its canonical metric. We write $\mathbb S^n\subset \mathbb R^{n+1}$ to mean we view it as the  unit sphere centered at the origin in the Euclidean space
$\mathbb R^{n+1}.$ 

We will  consider 3-manifolds which have nonnegative Ricci curvature and admit a conformal immersion into $\mathbb S^3.$ This includes the simply connected conformally flat manifolds of nonnegative Ricci curvature.  Recall that a Riemannian manifold 
$(M,g)$ of dimension $n$ is said to be conformally flat if each of its points has an open neighborhood which is conformally diffeomorphic to an open subset of $\mathbb R^n.$ In other words, $(M,g)$ is conformally flat if  it admits a coordinate covering 
$\{U_{\alpha}, \phi_{\alpha}\}$ with $\phi_{\alpha}: (U_{\alpha}, g)\longrightarrow 
\mathbb S^n$  conformal. When $n\geq 3$  and $M$ is simply connected, it is  a classical fact that this implies the existence of   a conformal immersion $(M,g)\longrightarrow \mathbb S^n$, the {\it developing map}, which is unique up to M\"{o}bius transformations. This is shown using Liouville's theorem and a standard monodromy argument. We  do not assume our manifolds to be complete. Assuming completeness, Zhu \cite{zhu} proved that a complete simply connected conformally flat $n$-manifold of nonnegative Ricci curvature is either conformally equivalent  to $\mathbb R^n$ or $\mathbb S^n,$ or is isometric to 
$\mathbb S^{n-1}\times \mathbb R.$ It is clear that small perturbations of the (standard) metric  of $\mathbb S^n$  in its conformal class produce metrics with positive Ricci curvature. Also, Zhu \cite{zhu} exhibited (rotationally symmetric) complete metrics on $\mathbb R^n$ of nonnegative Ricci curvature  which are   conformally  flat and non flat.  Our results apply in particular to these manifolds in dimension 3.

To get information about stable CMC (resp. index one minimal) surfaces, one usually  chooses suitable test functions and evaluates the quadratic form $Q$ on them. A very useful source of test functions  is the following slight extension of a result originally due to Hersch  and extended by Yang and Yau, and  by Li and Yau. 

\begin{lemma}\label{hersch} \cite{{hersch,L-Y, Y-Y}}
Let $\Sigma$ be a compact Riemannian surface admitting a branched conformal map $\Psi: \Sigma\longrightarrow
\mathbb{S}^n$ and  $\rho:\Sigma \longrightarrow \mathbb R$ a  smooth positive function. Then there exists a conformal diffeomorphism $T: \mathbb{S}^n\longrightarrow
\mathbb{S}^n$ such that the vector valued map $T\circ \Psi: \Sigma \longrightarrow \mathbb S^n\subset \mathbb R^{n+1}$ has mean value zero
\begin{equation*}
\int_{\Sigma} (T\circ \Psi) \rho\, dA=0,
\end{equation*}
where $dA$ is the volume element on $\Sigma.$
\end{lemma}

Lemma \ref{hersch} and its variants have been  widely used by several authors, in  various contexts related to  stability, for meromorphic maps $\Sigma\longrightarrow \mathbb S^2$ of controlled energy, see for instance \cite{ros0} and the references therein. Our idea here is to use it for conformal maps into $\mathbb S^3$ taking advantage of the existence of a conformal  immersion of the ambient manifold  into $\mathbb S^3$ and utilizing the   known fact that follows; we include a proof for completeness.

\begin{proposition}\label{conformal}
Let  $M$ be  an orientable Riemannian 3-manifold and $\Sigma$ a
compact orientable surface  without boundary immersed in $M.$
 Denote by $K_s$ the  sectional curvature of $M$ evaluated on the tangent plane to $\Sigma,$  by $H$ the mean curvature of $\Sigma$ and by $dA$ its area element. Then the quantity 
 \begin{equation}\label{willmore}
\int_{\Sigma} (H^2 + K_s)\,  dA 
\end{equation}
is invariant under conformal changes of the metric on $M.$
\end{proposition}

\begin{proof} Let $\kappa_1$ and $\kappa_2$ denote the principal curvatures of $\Sigma.$ It is straightforward (cf. relation (\ref {principal}) below) to check that the density $(\kappa_1-\kappa_2)^2\,dA$ is invariant under conformal changes of the metric on $M.$ Integrating over $\Sigma$ and using Gauss equation and Gauss-Bonnet formula gives the invariance of (\ref{willmore}).
\end{proof}

\section{Main result}
\medskip

We now state our main result. It improves a  weaker one we obtained previously (\cite{souam}, Theorem 3.3 (i)).

\begin{theorem}\label{the result} Let $M$ be a  Riemannian 3-manifold of nonnegative Ricci curvature, Ric $\geq 0.$ Suppose $M$ is simply connected and conformally flat or more generally that it admits a conformal immersion into $\mathbb S^3$. Let $\Sigma$ be a  compact orientable  surface without boundary immersed in $M$ with unit normal $N.$ Suppose $\Sigma$ is a stable CMC or an index one minimal surface. Then

\begin{enumerate}[\sf (i)]
\item If  $\Sigma$ is not connected then it is a finite union of totally geodesic surfaces and $\text{Ric}(N)=0.$ In this case $\Sigma$ is a stable minimal surface. 
\item If  $\Sigma$ is connected then 
it is  homeomorphic to either a sphere  or a torus . In the latter case, $\Sigma$ is minimal, embedded, has the conformal structure of a flat square torus and 
$\text{Ric}(N)=0.$ 

In particular, if $M$ has positive Ricci curvature, Ric $>0,$ then $\Sigma$ is topologically a sphere. 
\end{enumerate}
\end{theorem}
Before proving the theorem, we emphasize that we make no assumption about the completeness of $M.$ Also, since by assumption $M$ immerses into $\mathbb S^3,$ it is orientable and so the existence of a global unit normal $N$ to the surface $\Sigma$ is equivalent to its orientability.

\begin{proof}
We first treat the CMC case. The proof in the minimal case is similar and will be sketched below.

 If $\Sigma$ is not connected then one can take as a test function a function which is a non zero constant on each connected component and with mean value zero  and obtain immediately the conclusion of {\rm (i)} .
 
 To prove {\rm (ii)}, denote by $X: \Sigma\longrightarrow M$ the immersion of the CMC surface $\Sigma$ and by $F: M \longrightarrow \mathbb S^3\subset \mathbb R^4$ a conformal immersion (as we recalled in Section \ref{prelim}, such a map always exists when $M$ is conformally flat and simply connected). 
Set $\Psi=F\circ X.$ By Lemma \ref{hersch} applied to $\Psi$ with $\rho\equiv 1,$ we can assume 
 that 
\begin{equation*}
\int_{\Sigma} \Psi \, dA =0.
\end{equation*}
We therefore can use the coordinate functions of $\Psi$ as test functions. We thus have 
\begin{equation}\label{holom}
  0\leq Q(\Psi_i,\Psi_i)= \int_{\Sigma} |\nabla \Psi_i|^2 - (|\sigma|^2+\text{Ric}(N))\Psi_i^2 ,\quad\quad i=1,\dots,4.
  \end{equation}
Summing up these inequalities and taking into account that $|\Psi|=1,$  we get:
\begin{equation}\label{ineq1}
\int_{\Sigma}(|\sigma|^2 + \text{Ric}(N))\, dA  \leq \int_{\Sigma} |\nabla \Psi|^2 dA.
\end{equation}
Denote by $K$ and $K_s$ respectively the intrinsic curvature of $\Sigma$ and the sectional curvature of $M$ evaluated on the tangent plane to $\Sigma.$ By Gauss equation we have $|\sigma|^2 = 4H^2 + 2K_s-2K.$  So we rewrite  (\ref{ineq1}) as follows:
\begin{equation}\label{ineq2}
\int_{\Sigma}( 4H^2+2K_s-2K+\text{Ric}(N))\,dA \leq \int_{\Sigma} |\nabla \Psi|^2 dA.
\end{equation}
Call $g$ the standard metric on $\mathbb S^3.$ As $\Psi$ is conformal,  $\int_{\Sigma} |\nabla \Psi|^2 dA = \int_{\Sigma}  2\,\text{Jacobian}(\Psi)\, dA= 2\,\text{area}(\Sigma, \Psi^{\ast}g).$ Here $\text{area}(\Sigma, \Psi^{\ast}g)$ is the area of $\Sigma$ for the induced metric $\Psi^{\ast}g.$

Using Gauss-Bonnet formula, we transform inequality (\ref {ineq2}) into:
\begin{equation}\label{ineq3}
\int_{\Sigma}(2H^2+\text{Ric}(N)) dA + 2 \int_{\Sigma} (H^2+ K_s) \,dA - 4\pi \chi(\Sigma) \leq2\,\text{area}(\Sigma, \Psi^{\ast}g).
\end{equation}

Denote by $\bar{H}$ and $d\bar{A},$ respectively the mean curvature of the immersion $\Psi$ and the area element induced on $\Sigma.$ Using Proposition \ref{conformal}, we have:
\begin{equation}\label{ineq4}
\text{area}(\Sigma, \Psi^{\ast}g) \leq \int_{\Sigma} (\bar{H}^2 + 1)\,d\bar{A}= \int_{\Sigma} (H^2+ K_s) \,dA 
\end{equation}
Putting this into (\ref{ineq3}), we obtain:
\begin{equation}\label{ineq5}
\int_{\Sigma}(2H^2+\text{Ric}(N)) dA\leq 4\pi \chi(\Sigma)
\end{equation}

Since  $M$ has nonnegative Ricci curvature, this shows $\chi(\Sigma)\geq 0,$ that is,
$\Sigma$ is homeomorphic  either to a sphere or to a torus. 

Suppose $\Sigma$ is topologically a torus, i.e $\chi(\Sigma)=0.$ Then, as Ric $\geq 0,$  we have equality in   (\ref{ineq5}), $H=0$ and Ric$(N)=0.$ Furthermore  the equality is also reached in all  the intermediate inequalities and in particular in  (\ref{ineq4}). So $\bar H = 0,$ that is, $\Psi $ is a conformal minimal immersion into $\Bbb S^3.$ 
Again, as equality is reached in (\ref{ineq1}), we have:
\begin{equation}\label{area}
\int_{\Sigma} |\sigma|^2\, dA = \int_{\Sigma} |\nabla \Psi|^2\,dA= 2\,\text{area}(\Sigma, \Psi^{\ast}g) 
\end{equation}

We now prove $\Psi$ is an embedding, this will show  $X$ is an embedding too. As $\Sigma$ is a torus, there exists a meromorphic map $ \Sigma\longrightarrow \mathbb S^2$ of degree 2 which  can be taken, again by Lemma \ref{hersch}, such that $\int_{\Sigma} \phi \,dA = 0.$ Taking as test functions the coordinate functions $\phi_i,\, i=1,2,3,$ we have:
\begin{equation}\label{ineq6}
 \int_{\Sigma} |\sigma|^2\, \phi_i^2 \,dA\leq \int_{\Sigma} |\nabla \phi_i|^2 \,dA,\qquad i=1,2,3.
\end{equation}

Summing up these inequalities and taking into account (\ref{area}), we get:
\begin{equation} \label{ineq7}
2\,\text{area}(\Sigma, \Psi^{\ast}g) = \int_{\Sigma} |\sigma|^2 \,dA \leq \int_{\Sigma} |\nabla \phi|^2 \,dA= 16\pi.
\end{equation}

 Suppose $\Psi$ is not an embedding, then Li and Yau \cite{L-Y} have shown that in this case
  $ \text{area}(\Sigma, \Psi^{\ast}g) \geq 8\pi.$ So equality is achieved in (\ref {ineq7}) and hence also in (\ref{ineq6}). 

So the holomorphic map $\phi$ satisfies
$Q(\phi_i,\phi_i)=0,$ for $ i=1,2,3.$ As $\Sigma$ is stable, for any $v\in\mathcal{C}^\infty (\Sigma)$ satisfying $\int_{\Sigma} v=0$
and any $t\in\mathbb R$ we have $Q(\phi_i +tv,\phi_i+tv) \geq 0$ and so $Q(\phi_i,v)=0.$ It follows  that each of the  functions $\phi_i$
 satisfies the equation 
$$\Delta \phi_i + |\sigma|^2 \phi_i =c_i,$$
for some real constant $c_i,\, i=1,2,3.$ So $\phi$ verifies the equation:
\begin{equation}\label{phi1}
\Delta \phi + |\sigma|^2 \phi=\vec c
\end{equation}
where $\vec c=(c_1,c_2,c_3).$ 

On the other hand, since  $\phi: \Sigma \longrightarrow \mathbb S^2$ is holomorphic it is harmonic and therefore:
\begin{equation}\label{phi2}
\Delta \phi + |\nabla \phi|^2 \phi=0.
\end{equation}
As $\phi$ is non constant and $|\phi|=1,$ it follows easily from (\ref{phi1})
 and (\ref{phi2}) that necessarily $\vec c=\vec 0$ and $|\sigma|^2=|\nabla\phi|^2.$ So the Jacobi operator of
$\Sigma$ writes as $L= \Delta +|\nabla\phi|^2$ and the stability assumption implies that $L$ has only one negative eigenvalue.
 Otherwise said the holomorphic map $\phi$
 has index one. However such maps do not exist on tori (cf.  \cite{ros2}), a contradiction. 
 Therefore $\Psi$ is a conformal minimal embedding of a  torus in $\mathbb S^3.$ By the recent solution by Brendle \cite{brendle} to Lawson's conjecture, we know $\Psi(\Sigma)$ is congruent to the Clifford torus. This shows  $\Sigma$ is conformal to a square flat torus.
 
  Assume now  $\Sigma$ is an index one minimal surface. Let $\varphi_1\in\mathcal C^{\infty}(\Sigma)$ be a non-trivial first eigenfunction of the Jacobi operator $L.$ The index one hypothesis means that 
  \begin{equation*}\label{index1}
Q(u,u)\geq 0,\quad \text{for any}\quad  u\in  \mathcal  {C}^\infty(\Sigma) \quad \text{satisfying} \int_{\Sigma} u\varphi_1\,dA =0.
\end{equation*}
   It is well known that $\varphi_1$ has no zeros and can thus be taken positive. We then apply, as above, Lemma \ref{hersch} to the map $\Psi$ with $\rho=\varphi_1.$ The rest of the proof is similar, we omit the details.
\end{proof}

\begin{remark}
Possibility (i) in Theorem \ref{the result} happens for instance in the Riemannian product $\mathbb S^2\times\mathbb R$ 
where a finite union of slices $\mathbb S^2\times \{t\}$ is a stable compact minimal surface.
\end{remark}

We end this section with an application to the isoperimetric problem as announced in the introduction. Let $M$ be a compact Riemannian 3-manifold. From  Geometric Measure Theory results, one knows  that for any $0<V< \text{Vol}(M)$ there exists a (possibly disconnected) regular domain $\Omega\subset M$ of volume $V$ such that $\partial \Omega$ minimizes the area among all regular domains in $M$ of volume $V.$ Moreover 
$\partial \Omega$ is a (possibly disconnected) regular CMC surface and is stable. Solving the isoperimetric problem in $M$ consists in determining the isoperimetric regions. This is known in very few cases. See \cite{ros1} for an account on the subject. Ros \cite{ros0, ros1} conjectured that the isoperimetric regions in the 3-sphere endowed with a metric of positive Ricci curvature are topologically 3-balls. As a consequence of Theorem \ref{the result}, we have the following  positive  partial answer to Ros' conjecture

\begin{corollary}\label{ros-conjecture} Let $M$ be the 3-sphere endowed with a metric of positive Ricci curvature such that $M$ is conformally flat. Then 
the isoperimetric regions in $M$ are topologically 3-balls.
\end{corollary}

Note that by the existence of the developing map we recalled in Sect. 2, the conformal flatness of $M$ means that $M$ is conformally equivalent to the standard sphere $\mathbb S^3$ (Kuiper's theorem). 

\section{Examples of metrics with stable tori}\label{examples}
\medskip

 We show in this section that Theorem \ref{the result} is optimal. We will perturb  the standard  metric on the 3-sphere in its conformal class to obtain metrics of nonnegative Ricci curvature having minimal tori which are stable as CMC surfaces (and have index one as minimal surfaces). We will do this keeping the Clifford torus minimal and making it stable for the new metrics. 

We start with the following  parametrization (see \cite{kapouleas-yang}) of the unit sphere centered at the origin $\mathbb S^3 \subset \mathbb R^4\simeq \mathbb C^2$ 

\begin{equation*}
\Phi: \mathbb R^2\times \left(-\frac{\pi}{4},\frac{\pi}{4}\right)\longrightarrow \mathbb S^3
\end{equation*}
\begin{align*}
\Phi(\theta,\phi,t) = &\sin \left( t+\frac{\pi}{4}\right) (\cos(\sqrt  2\, \theta), \sin(\sqrt  2\, \theta), 0,0) \\
      &+\cos \left( t+\frac{\pi}{4}\right) (0,0,\cos(\sqrt  2\,  \phi), \sin(\sqrt  2\, \phi))
\end{align*}

This parametrization covers the unit sphere with two orthogonal circles removed, namely $\mathbb S^3 \setminus\{(z_1,z_2)\in \mathbb C^2, \quad z_1=0 
\quad \text{or} \quad z_2=0\}.$ 

Note that 
$\Phi(.,.,0): \mathbb R^2\longrightarrow \mathbb S^3$ parametrizes the Clifford torus 
\begin{equation*} \mathbb T=\left\{(z_1,z_2)\in\mathbb C^2: \quad |z_1|^2=|z_2|^2=\frac{1}{\sqrt 2}\right\}.
\end{equation*}
and, for each $t\in \left(-\frac{\pi}{4},\frac{\pi}{4}\right),$ the map $\Phi(.,.,t)$ parametrizes the  surface  parallel to $\mathbb T$  at signed distance $t$.

The expression of the standard spherical metric in these (local)  coordinates is
\begin{equation*}
g= (1+  \sin(2t)) d\theta^2 + (1- \sin(2t)) d\phi^2 + dt^2.
\end{equation*}
We modify conformally the metric $g$ into $\bar g = e^{2\varphi} g,$ where $\varphi$ is a smooth function of $t$ alone to be chosen below.
The Ricci curvature tensor $\overline{\text {Ric}}$ of $\bar g$ is related to the Ricci tensor, Ric,  of $g$ as follows (cf. \cite{besse} p. 59, the sign convention for the Laplacian is opposite to ours)
\begin{equation*}\label{ricci}
\overline{\text{Ric}} = {\text {Ric}} -\left(\nabla d\varphi -d\varphi\otimes d\varphi \right) -\left(\Delta \varphi +|\nabla \varphi|^2\right) g
\end{equation*}
Straightforward computations give
\begin{equation*} \nabla d\varphi\left(\frac{\partial}{\partial t},\frac{\partial}{\partial t}\right)= \varphi^{\prime\prime},\,  \nabla d\varphi\left(\frac{\partial}{\partial \theta},\frac{\partial}{\partial \theta}\right)= \varphi^\prime \cos (2t),\,  \nabla d\varphi\left(\frac{\partial}{\partial \phi},\frac{\partial}{\partial \phi}\right)= -\varphi^\prime \cos(2t)
\end{equation*}
\begin{equation*} \nabla d\varphi\left(\frac{\partial}{\partial t},\frac{\partial}{\partial \theta}\right)= 
 \nabla d\varphi\left(\frac{\partial}{\partial \theta},\frac{\partial}{\partial \phi}\right)= 
  \nabla d\varphi\left(\frac{\partial}{\partial \phi},\frac{\partial}{\partial t}\right)= 0
\end{equation*}
and
\begin{equation*} \Delta \varphi=\varphi^{\prime\prime} - 2 \frac{\sin(2t)}{\cos(2t)} \varphi^\prime.
\end{equation*}

Taking into account that $\text{Ric}=2g,$ it follows that
\begin{equation}\label{ricci}
\begin{cases}\overline{\text{Ric}} \left(\frac{\partial}{\partial t},\frac{\partial}{\partial t}\right)=  2(1-\varphi^{\prime\prime}) + 2\, \frac{\sin (2t)}{\cos (2t)}\, \varphi^\prime\\
\overline{\text{Ric}} \left(\frac{\partial}{\partial \theta},\frac{\partial}{\partial \theta}\right)= (1+\sin(2t))\left(2 - \varphi^{\prime\prime} +2\,\frac{\sin (2t)}{\cos (2t)}\varphi^\prime -(\varphi^\prime)^2\right)-\varphi^\prime\cos(2t)\\
\overline{\text{Ric}}\left(\frac{\partial}{\partial \phi},\frac{\partial}{\partial \phi}\right)=(1-\sin(2t))\left(2 - \varphi^{\prime\prime} +2\,\frac{\sin (2t)}{\cos (2t)}\varphi^\prime -(\varphi^\prime)^2\right)+\varphi^\prime\cos(2t)\\
\overline{\text{Ric}} \left(\frac{\partial}{\partial \theta},\frac{\partial}{\partial \phi}\right)= \overline{\text{Ric}} \left(\frac{\partial}{\partial \phi},\frac{\partial}{\partial t}\right)=
\overline{\text{Ric}} \left(\frac{\partial}{\partial t},\frac{\partial}{\partial \theta}\right)=0.
\end{cases}
\end{equation}

 Let now $r\in (0,\frac{\pi}{8})$ and take a smooth {\it even} function  $\zeta: (-\frac{\pi}{4},\frac{\pi}{4})\longrightarrow\mathbb R$ satisfying the following conditions
 
 \begin{enumerate} [\sf (i)]
 \item $\zeta(0)=1, \zeta$ is  decreasing on $(0,r)$ and  $\zeta(r)=0$
 \item $ \zeta$ is negative on $(r,2r)$
 \item $\zeta\equiv 0$ on $[2r, \frac{\pi}{4})$ 
 \item $\int_{0}^{2r} \zeta(s)\,ds=0.$ 
 \end{enumerate}
  We set $\varphi(t)=\int_0^t(\int _0^s \zeta(x)\,dx)  \,ds.$ 
  
   Denote by $\kappa_1, \kappa_2$ the principal curvatures of a surface with unit normal $N$ for the 
  metric $g.$ Its  principal curvatures  for the metric $\bar g$ are given by
\begin{equation}\label{principal}  \bar\kappa_i = e^{-2\varphi} (\kappa_i + \langle \nabla \varphi, N\rangle ) , \qquad i=1,2
\end{equation}
A unit normal to the Clifford torus $\mathbb T$ is $N= \frac{\partial}{\partial t}_{|t=0}$ and so
\begin{equation*}
\bar\kappa_i=e^{-2\varphi(0)} (\kappa_i + \varphi^\prime(0) ), \qquad i=1,2
\end{equation*}

With the choice of $\varphi$ as above, we have $\varphi(0)=0,$ so that the metric on  $\mathbb T$ is unchanged and $\bar N =N$ is a unit normal to $\mathbb T$ 
for the metric $\bar g.$  As $ \varphi^\prime (0)=0$ we have $\bar\kappa_i=\kappa_i,\, i=1,2.$ 
The Clifford torus $\mathbb T$ is therefore still minimal for the metric $\bar g$ and the square of the norm of its second fundamental form $\bar{\sigma}$ is $|\bar{\sigma}|^2=|\sigma|^2=2.$
Since  
 $ \varphi^{\prime\prime}(0)=1,$ we have $\overline{\text{Ric}} (\bar N, \bar N)=0.$ So the Jacobi operator of $\mathbb T$  for the new metric is $L=\Delta + 2.$ For the Clifford torus, the first non zero eigenvalue of $\Delta$ is $\lambda_1=2.$ It follows that $\mathbb T$ is stable as a CMC surface and has index one as a minimal surface for the metric $\bar g.$ 
 
 For $|t|\geq 2r,$ using  {\rm (iii)} and  {\rm (iv)}, we have $\bar g= C\, g$ where $C=e^{2\varphi(2r)}$ is a constant and so $\bar g$ extends to the whole of $\mathbb S^3$. It remains to show that we can choose $r$ so that, with $\varphi$ as above, the Ricci curvature of $\bar g$ is nonnegative. 
 
 It follows from   {\rm (i)}, {\rm (ii)}, {\rm (iii)} and {\rm (iv)} that $\varphi^\prime(t)=\int_0^t \zeta(s)\,ds\geq 0$  on $[0,\frac{\pi}{4})$ (resp. $\leq 0$ on $(-\frac{\pi}{4},0]$) and so,  for any $t\in(-\frac{\pi}{4},\frac{\pi}{4}),\quad \sin(2t)\varphi^{\prime} (t)\geq 0.$  Observe also that, by   properties {\rm (i), (ii)} and  {\rm (iii)},  we have   $|\varphi^\prime|\leq r$  and  $\varphi^{\prime\prime}\leq 1$ on $(-\frac{\pi}{4},\frac{\pi}{4}).$ Using this one checks easily from (\ref{ricci}) that  $\overline{\text{Ric}} \geq 0$ if  $r$ is taken small enough.
   It is interesting to note that, with the choice of $r$ small and $\varphi$ as above, $\overline{\text{Ric}}$ vanishes only in the direction of $ \frac{\partial}{\partial t}_{|t=0}.$

\bibliographystyle{alpha}

\begin{thebibliography}{06}


\bibitem{barbosa et al} J-L. Barbosa; M.  do Carmo \& J. Eschenburg.: {\em 
 Stability of hypersurfaces of constant mean
curvature in Riemannian manifolds. } Math. Z.  197 
(1988),  no. 1, 123-138.


\bibitem{besse} A. Besse.: {\em  Einstein manifolds. }Ergebnisse der Mathematik und ihrer Grenzgebiete (3), 10. Springer-Verlag, Berlin, 1987.

\bibitem{brendle}  S. Brendle.: {\em Embedded minimal tori in $\mathbb S^3$ and the Lawson conjecture.} Acta Math. (to appear).


\bibitem{hersch} J. Hersch.: {\em 
Quatre propri{\'e}t\'es isop\'erim\'etriques de membranes
sph\'eriques homog\`enes. }
C. R. Acad. Sci. Paris SŽr. A-B 270 (1970),
A1645- A1648.


\bibitem{kapouleas-yang} N. Kapouleas \& S-D. Yang.:{\em  Minimal surfaces in the three-sphere by doubling the Clifford torus.} Amer. J. Math. 132 (2010), no. 2, 257- 295.

\bibitem{L-Y} P. Li \& S-T. Yau.: {\em  A new conformal invariant
and its applications to the Willmore conjecture and
the first eigenvalue of compact surfaces. } Invent.
Math.  69  (1982), no. 2, 269- 291.


\bibitem{pitts} J.T. Pitts.: {\em  Existence and regularity of minimal surfaces on Riemannian manifolds. }
Mathematical Notes, 27. Princeton University Press, Princeton, N.J.; University of Tokyo Press, Tokyo, 1981. iv+330 pp.

\bibitem{pitts-rubinstein}  J.T.  Pitts \& J.H. Rubinstein.: { Equivariant minimax and minimal surfaces in geometric three-manifolds. } Bull. Amer. Math. Soc. (N.S.) 19 (1988), no. 1, 303- 309.

\bibitem{ritore} M. Ritor\'e.:  {\em Index one minimal surfaces in flat
three space forms.}  Indiana Univ. Math. J.  46 
(1997),  no. 4, 1137- 1153.

\bibitem{ritore-ros} M. Ritor\'e \& A. Ros.: {\em The spaces of index one minimal surfaces and stable constant mean curvature surfaces embedded in flat three manifolds. } Trans. Amer. Math. Soc. 348 (1996), no. 1, 391- 410.

\bibitem{ros0} A. Ros.: {\em The isoperimetric and Willmore problems.}  Global differential geometry: the mathematical legacy of Alfred Gray (Bilbao, 2000), 149- 161, Contemp. Math., 288, Amer. Math. Soc., Providence, RI, 2001.

\bibitem{ros1} A. Ros.: {\em The isoperimetric problem.} Global theory
of minimal surfaces,  175- 209, Clay Math. Proc., 2,
Amer. Math. Soc., Providence, RI, 2005.


\bibitem{ros2} A. Ros.: {\em  One-sided complete stable minimal surfaces. } J. Differential Geom.  74  (2006),  no. 1,
69- 92. 



\bibitem{souam} R. Souam.:  {\em  On stable constant mean curvature surfaces in $\mathbb S^2 \times \mathbb R$ and $\mathbb H^2 \times \mathbb R.$ }   Trans. Amer. Math. Soc. 362 (2010), 2845- 2857.
 
 \bibitem{Y-Y} P.C. Yang and S-T. Yau.:{\em Eigenvalues of the Laplacian of compact Riemann surfaces and minimal submanifolds. } Ann. Scuola Norm. Sup. Pisa Cl. Sci. (4) 7 (1980), no. 1, 55Ð63.
 
 \bibitem{zhu} S-H. Zhu.: {\em  The classification of complete locally conformally flat manifolds of nonnegative Ricci curvature.} Pacific J. Math. 163 (1994), no. 1, 189- 199.
 
\end{thebibliography}

\enddocument